\documentclass[11pt]{article}
\usepackage{amsmath}
\usepackage{amsthm}
\usepackage{amssymb,amsfonts}
\usepackage{hyperref}
\usepackage{latexsym}
\usepackage{todonotes}
\usepackage{nicefrac}
\usepackage{epsfig}
\usepackage{stmaryrd}
\usepackage{setspace}
\usepackage{enumerate}
\usepackage[all]{xypic}
\usepackage{bbm,ifpdf,tikz}
\ifpdf
\usepackage{pdfsync}
\fi

\newtheorem{theorem}{Theorem}[section]

\newtheorem{lemma}[theorem]{Lemma}
\newtheorem{proposition}[theorem]{Proposition}
\newtheorem{corollary}[theorem]{Corollary}

{
\theoremstyle{definition}

\newtheorem{example}[theorem]{Example}

\newtheorem{remark}[theorem]{Remark}
}

\newcommand{\excise}[1]{}

\renewcommand{\dim}{\operatorname{dim}}

\renewcommand{\and}{\qquad\text{and}\qquad}
\newcommand{\Ind}{\operatorname{Ind}}

\newcommand{\Hom}{\operatorname{Hom}}
\newcommand{\triv}{\operatorname{triv}}

\newcommand{\N}{\mathbb{N}}
\newcommand{\C}{\mathbb{C}}

\newcommand{\IH}{I\! H}

\newcommand{\cA}{\mathcal{A}}
\newcommand{\la}{\lambda}

\newcommand{\OS}{S}

\renewcommand{\cL}{\mathcal{L}}

\newcommand{\FS}{\operatorname{FS}}

\newcommand{\FSB}{\operatorname{FS}_{\!B}}
\newcommand{\FSA}{\operatorname{FS}_{\!A}}
\newcommand{\op}{\operatorname{op}}
\newcommand{\FSAop}{\FSA^{\op}}
\newcommand{\FSBop}{\FSB^{\op}}
\newcommand{\OSBop}{\OSB^{\op}}
\newcommand{\FCop}{\FC^{\op}}
\newcommand{\Rep}{\operatorname{Rep}}
\newcommand{\FC}{\operatorname{C}}
\newcommand{\OSB}{\operatorname{OS}_{\!B}}
\newcommand{\init}{\operatorname{init}}
\renewcommand{\O}{\operatorname{O}}

\begin{document}
\spacing{1.2}
\noindent{\Large\bf A type B analogue of the category of finite sets with surjections}\\

\noindent{\bf Nicholas Proudfoot\footnote{Supported by NSF grants DMS-1565036 and DMS-1954050.}}\\
Department of Mathematics, University of Oregon,
Eugene, OR 97403\\

{\small
\begin{quote}
\noindent {\em Abstract.}
We define a type $B$ analogue of the category of finite sets with surjections, and we study the representation theory of this category.
We show that the opposite category is quasi-Gr\"obner, which implies that submodules of finitely generated modules are again finitely generated.
We prove that the generating functions of finitely generated modules have certain prescribed poles, and we obtain
restrictions on the representations of type $B$ Coxeter groups that can appear in such modules.  Our main example is a module
that categorifies the degree $i$ Kazhdan--Lusztig coefficients of type $B$ Coxeter arrangements.
\end{quote} }

\section{Introduction}
Let $\FSA$ be the category whose objects are nonempty finite sets and whose morphisms are surjective maps.
The $A$ in the subscript is there to call attention to the fact that this is a ``type $A$'' structure.  More concretely,
for any positive integer $n$, the automorphism group of the object $[n] = \{1,\ldots,n\}$ is the Coxeter group type $A_{n-1}$,
and the set of equivalence classes of morphisms with source $[n]$ may be identified with the set of flats of the Coxeter hyperplane arrangement
of type $A_n$ (Example \ref{FS-flats}).  Our aim is to define and study a ``type $B$'' analogue of this category, which we call $\FSB$.

We begin with the definition.  An object of $\FSB$ is a pair
$(E,\sigma)$, where $E$ is a finite set and $\sigma:E\to E$ is an involution with a unique fixed point.
A morphism from $(E_1,\sigma_1)$ to $(E_2,\sigma_2)$ is a surjective map $\varphi:E_1\to E_2$ with
$\varphi\circ\sigma_1 = \sigma_2\circ \varphi$.  
For any natural number $n$, we write $[-n,n]$
to denote the object given by the set of integers between $-n$ and $n$ (inclusive) and the involution
$k\mapsto -k$; every object of $\FSB$ is isomorphic to $[-n,n]$ for some $n\in\N$. 
The automorphism group $W_n$ of the object $[-n,n]$ is the Coxeter group of type $B_n$, and the set of equivalence classes
of morphisms with source $[-n,n]$ may be identified with the set of flats of the Coxeter hyperplane arrangement
of type $B_n$ (Example \ref{FS-flats}).

\begin{remark}
A more naive definition of $\FSB$ would be to take finite sets with free involutions and equivariant maps.  
This category would have the right automorphism
groups, but it would not have the same relationship with flats of the Coxeter hyperplane arrangements of type $B$.
This distinction is not relevant when one studies the type $B$ analogue of finite sets with {\em injections} \cite{Wilson}, since any equivariant injection
would have to preserve the fixed point.
\end{remark}

\begin{remark}
It is natural to ask why we do not also introduce and study a ``type $D$'' analogue of this category.
The brief answer is that the classes of Coxeter arrangements of types $A$ and $B$ are closed under contraction (Examples \ref{FS-flats} 
and \ref{FSB-flats}), but the analogous statement is false in type $D$.  This property is crucial to the examples that we consider in this paper.
\end{remark}

For the remainder of the introduction, we describe the results for $\FSA$ and $\FSB$ in parallel for comparison.
All results that we state for $\FSA$ appear in either \cite{sam} or \cite{fs-braid}.

\subsection{Finiteness}
The first half of this paper is devoted to applying the Sam--Snowden Gr\"obner theory of combinatorial categories \cite{sam} to 
the opposite category $\FSBop$.
More concretely, we fix a left Noetherian ring $k$ and an essentially small category $\FC$ (which will always be either $\FSA$ or $\FSB$)
and study the category $\Rep_k(\FCop)$ of contravariant functors from $\FC$
to the category of left $k$-modules.  Such a functor is called an $\FCop$-module over $k$.  
Given an object $x$, the {\bf principal projective} $P_x\in \Rep_k(\FCop)$ is the module that assigns to an object $y$
the free $k$-module with basis $\Hom_{\FC}(y,x)$, with maps defined on basis elements by composition. 
A module $M$ is called {\bf finitely generated}
if there exists a finite set of objects $x_1,\ldots,x_r$ and a surjective map from $\oplus_i P_{x_i}$ to $M$.
The following theorem of Sam and Snowden says that finitely generated $\FSAop$-modules form an Abelian category
\cite[Theorem 8.1.2]{sam}.

\begin{theorem}\label{qGA}
Any submodule of a finitely generated $\FSAop$-module over $k$ is itself finitely generated.
\end{theorem}

We prove here the analogous theorem for $\FSB$.

\begin{theorem}\label{qGB}
Any submodule of a finitely generated $\FSBop$-module over $k$ is itself finitely generated.
\end{theorem}

An $\FSAop$-module $M$ is called {\bf finitely generated
in degree \boldmath{$\leq d$}} if the generating objects can all be taken to be sets of cardinality at most $d$.  Similarly, an $\FSB$-module $N$
is called {\bf finitely generated in degree \boldmath{$\leq d$}} if the generating objects can all be taken to have at most $d$ free orbits;
equivalently, they can all be taken to be objects of the form $[-n,n]$ with $n\leq d$.  A module over either category is called 
{\bf \boldmath{$d$}-small} if it is isomorphic to a subquotient of a module that is finitely generated in degree $\leq d$.  Theorems \ref{qGA} 
and \ref{qGB} immediately implies that a $d$-small object is itself finitely generated, though the degree of generation
might be much larger than $d$.

Borrowing terminology from \cite{PR-trees} and \cite{PR-genus}, we call a module 
{\bf \boldmath{$d$}-smallish} if it admits a filtration whose associated graded module is $d$-small.
The motivation for this definition is that, if we have a spectral sequence converging to $N$ for which the modules on the $E_1$-page are all 
$d$-small, the same will necessarily be true for the $E_\infty$-page, which is isomorphic to the associated graded module of $N$ with respect
to some filtration, and $N$ is therefore $d$-smallish.  It is easy to prove that a $d$-smallish module is finitely generated \cite[Proposition 2.14]{PR-trees}.
We do not know whether or not a $d$-smallish module must be $d$-small. 

\subsection{Growth}
Fix a field $k$ of characteristic zero.
If $\lambda = (\la_1,\ldots,\la_{\ell(\la)})$ is a partition of $n$, we write $V_{\la}$ to denote the corresponding
irreducible representation of $S_n$ over $k$.  If $\la$ and $\mu$ are partitions with $|\la|+|\mu|=n$, we write $V_{\la,\mu}$ to denote
the corresponding irreducible representation of $W_n$ over $k$.

For an $\FSAop$-module $M$ and a positive integer $n$, we write $M[n]$ to denote the $S_n$-representation
$M([n])$, and we define the 
generating function
$$H_{\!A}(M; t) := \sum_{n=1}^\infty t^n \dim M[n].$$ 
If $M$ is $d$-smallish, we define the limit $$r_{\!A}^d(M) := \lim_{n\to\infty}\frac{\dim M[n]}{d^n},$$
which we will show always exists.
The following theorem was proved in \cite[Theorem 4.1]{fs-braid}.

\begin{theorem}\label{A-small}
Let $M$ be a $d$-smallish $\FSAop$-module.
\begin{enumerate}
\item The generating function $H_{\!A}(M; t)$
is a rational function whose poles are contained in the set $\{1/j \mid 1\leq j \leq d\}$.
\item The limit $r_{\!A}^d(M)$ exists.  Equivalently, $H_{\!A}(M; t)$ has at worst a simple pole at $1/d$, and $r_{\!A}^d(M)$
is the residue.
\item If $|\la|=n$ and $\Hom_{S_n}\!\big(V_\la, M[n]\big) \neq 0$, then $\ell(\la)\leq d$.  
\end{enumerate}
\end{theorem}

We now state the type $B$ analogue of Theorem \ref{A-small}.
For an $\FSBop$-module $N$ and a nonnegative integer $n$, we write $N[-n,n]$ to denote the $W_n$-representation
$N([-n,n])$, and we define the 
generating function $$H_{\!B}(N; t) := \sum_{n=0}^\infty t^n \dim N[-n,n].$$
If $N$ is $d$-smallish, we define the limit $$r_{\!B}^d(N) := \lim_{n\to\infty}\frac{\dim N[-n,n]}{(2d+1)^n},$$
which we will show always exists.

\begin{theorem}\label{B-small}
Let $N$ be a $d$-smallish $\FSBop$-module.
\begin{enumerate}
\item The generating function $H_{\!B}(M; t)$
is a rational function whose poles are contained in the set $\{1/j \mid 1\leq j \leq 2d+1\}$.
\item The limit $r_{\!B}^d(N)$ exists.  Equivalently, 
$H_{\!B}(N; t)$ has at worst a simple pole at $1/(2d+1)$, and $r_{\!B}^d(N)$ is the residue.
\item If $|\la|+|\mu|=n$ and $\Hom_{W_n}\!\big(V_{\la,\mu}, N[-n,n]\big) \neq 0$, then $\ell(\la)\leq d+1$ and $\ell(\mu)\leq d$.  
\end{enumerate}
\end{theorem}

\subsection{Examples}
For any nonempty finite set $E$, we define in Example \ref{type A arrangement} a hyperplane arrangement $\cA_E$ 
with the property that $\cA_{[n]}$ is the Coxeter arrangement of type $A_n$.  Similarly, for any object $(E,\sigma)$ of $\FSB$,
we define in Example \ref{type B arrangement} a hyperplane arrangement $\cA_{(E,\sigma)}$ 
with the property that $\cA_{[-n,n]}$ is the Coxeter arrangement of type $B_n$.

In Section \ref{sec:OS}, we define
an $\FSA$-module $\OS_{\!A}^i$ that takes $E$ to the degree $i$ part of the Orlik--Solomon algebra of $\cA_E$;
by taking the linear dual, we obtain an $\FSAop$-module $(\OS_{\!A}^i)^*$.  Similarly, we define
an $\FSB$-module $\OS_{\!B}^i$ that takes $(E,\sigma)$ to the degree $i$ part of the Orlik--Solomon algebra of $\cA_{(E,\sigma)}$
and the dual $\FSBop$-module $(\OS_{\!B}^i)^*$.
The following proposition was proved in \cite[Proposition 5.1]{fs-braid}.

\begin{proposition}\label{OSA-prop}
The $\FSAop$-module $(\OS^0_{\!A})^*$ is $1$-small.  For all $i>0$, the $\FSAop$-module $(\OS^i_{\!A})^*$ is $2i$-small, and 
$$r_{\!A}^{2i}\!\left((\OS^i_{\!A})^*\right)=0.$$
\end{proposition}

Here we prove the following type $B$ analogue of Proposition \ref{OSA-prop}.

\begin{proposition}\label{OSB-prop}
The $\FSBop$-module $(\OS^0_{\!B})^*$ is $0$-small.  For all $i>0$, the $\FSBop$-module$(\OS^i_{\!B})^*$ is $(2i-1)$-small,
and $$r_{\!B}^{2i-1}\!\left((\OS^i_{\!B})^*\right)=0.$$
\end{proposition}

\begin{remark}
The smallness shift between Propositions \ref{OSA-prop} and \ref{OSB-prop} (which we will see again in Theorems \ref{KLA-thm} and \ref{KLB-thm})
can be blamed on the fact that the object $[n]$ of $\FSA$ corresponds to the Coxeter group and Coxeter arrangement of type $A_{n-1}$, 
while the objet $[-n,n]$ of $\FSB$ corresponds to the Coxeter group and Coxeter arrangement of type $B_n$.
It is also related to the fact that $[1]$ is the terminal object of $\FSA$ while $[0,0]$ is the terminal object of $\FSB$.
\end{remark}

For any hyperplane arrangement $\cA$, one may define a singular algebraic variety $X_\cA$ called the {\bf reciprocal plane} of $\cA$.
This variety has vanishing intersection cohomology in odd degree, and the even degree intersection cohomology Poincar\'e polynomial
coincides with the {\bf Kazhdan--Lusztig polynomial} of the associated matroid \cite[Proposition 3.12]{EPW}.
In Section \ref{sec:KL}, we define an $\FSA$-module $D^i_{\!A}$ that takes a nonempty finite set $E$ to 
$\IH^{2i}\!\left(X_{\cA_E}\right)$ and an $\FSB$-module $D^i_{\!B}$ that takes an object $(E,\sigma)$ to 
$\IH^{2i}\big(X_{\cA_{(E,\sigma)}}\big)$.
One can think of $D^i_{\!A}$ and $D^i_{\!B}$ as categorifications of the degree $i$ Kazhdan--Lusztig coefficients of Coxeter arrangements
in types $A$ and $B$, respectively.
The following theorem was proved in \cite[Theorem 6.1]{fs-braid}.

\begin{theorem}\label{KLA-thm}
For any $i>0$, the $\FSAop$-module $\left(D^i_{\!A}\right)^*$ is $2i$-smallish,\footnote{In the published version of the paper, we claimed
that the module was $2i$-small, but we only proved that it is $2i$-smallish.  This mistake was corrected in the arXiv version.}
and we have
$$r_{2i}\!\left(\left(D^i_{\!A}\right)^*\right) = \frac{\dim D^{i-1}_{\!A}[2i]}{|S_{2i}|} = \frac{\dim D^{i-1}_{\!A}[2i]}{(2i)!}.$$
\end{theorem}

Here we prove the following type $B$ analogue of Theorem \ref{KLA-thm}.

\begin{theorem}\label{KLB-thm}
For any $i>0$, the $\FSBop$-module $\left(D^i_{\!B}\right)^*$ is $(2i-1)$-smallish, and we have
$$r_{2i-1}\!\left(\left(D^i_{\!B}\right)^*\right) = \frac{\dim D^{i-1}_{\!B}[1-2i,2i-1]}{|W_{2i-1}|} = \frac{\dim D^{i-1}_{\!B}[1-2i,2i-1]}{2^{2i-1}(2i-1)!}.$$
\end{theorem}

\vspace{\baselineskip}
\noindent
{\em Acknowledgments:}
This work benefited greatly from the efforts of Patrick Durkin, who helped to formulate 
the definition of $\FSB$ and wrote the first draft of the material in Section 3.
The author is also grateful to Eric Ramos for his valuable help and suggestions.

\section{Gr\"obner and \boldmath{$\O$}-lingual categories}
We begin by reviewing the relevant machinery from \cite{sam} that we will need to prove Theorems \ref{qGB} and \ref{B-small}.  
Let $\FC$ be an essentially small category.
Given morphisms $\varphi:x\to y$ and $\varphi':x\to y'$, we say $\varphi\leq \varphi'$
if there exists a morphism $\psi:y\to y'$ with $\varphi' = \psi\circ \varphi$.  If $\varphi\leq \varphi'\leq \varphi$, 
then $\varphi$ and $\varphi'$ are said to be {\bf equivalent}.
The poset of equivalence classes of morphisms out of $x$ is denoted $|\FC_x|$.  

We say that $\FC$ is {\bf directed} if it has no endomorphisms other than the identity maps.
We say that $\FC$ has {\bf property (G1)} if, for every object $x$,
there exists a well order $\prec$ on $\FC_x$ that with the property that $\varphi\prec \varphi' \Rightarrow \psi\circ \varphi\prec \psi\circ \varphi'$ whenever both compositions make sense.  We say that $\FC$ has {\bf property (G2)} if, for every object $x$, the poset $|\FC_x|$ is {\bf Noetherian}, meaning 
that every ideal (upwardly closed subset) has only finitely many minimal elements.
A directed category with properties (G1) and (G2) is called {\bf Gr\"obner}.

A functor $\Phi:\FC\to\FC'$ has {\bf property (F)} if, for any object $x$ of $\FC'$, there exist finitely many objects 
$y_1,\ldots,y_s$ of $\FC$ and morphisms
$\varphi_i:x\to \Phi(y_i)$ such that for any object $y$ of $\FC$ and any morphism
$\varphi:x \to \Phi(y)$ in $\FC$, there exists a morphism $\psi:y_i \to y$ in $\FC$ with $\varphi = \Phi(\psi)\circ \varphi_i$.
This definition is engineered precisely so that the following result will hold \cite[Propositions 3.2.3]{sam}.

\begin{proposition}\label{generators}
Suppose that $\Phi:\FC\to\FC'$ has property (F).
Suppose that $N\in\Rep_k(\FC')$ is finitely generated, with generating objects $x_1,\ldots,x_r$.
For each $1\leq i\leq r$, choose objects $y_{i1},\ldots,y_{is_i}$ of $\FC$ corresponding to $x_i$ as in the definition of property (F).
Then the module $\Phi^*N\in \Rep_k(\FC)$ is finitely generated, with generating objects $\{y_{ij}\mid 1\leq i\leq r, 1\leq j\leq s_i\}$.
\end{proposition}

The category $\FC'$ is called {\bf quasi-Gr\"obner} if there exists a Gr\"obner category $\FC$ and an essentially surjective functor
$\Phi:\FC\to\FC'$ with property (F).  In this case, the category $\FC$ is said to be a {\bf Gr\"obner cover} of $\FC'$.
Sam and Snowden use Proposition \ref{generators} to prove the following result \cite[Theorem 4.3.2]{sam}.

\begin{theorem}\label{SSqG}
If $\FC'$ is quasi-Gr\"obner and $k$ is a left Noetherian ring, 
then any submodule of a finitely generated $\FC'$-module over $k$ is itself finitely generated.
\end{theorem}

Given a finite set $\Sigma$, we denote the set of words (finite sequences) in $\Sigma$ by $\Sigma^\star$.  A {\bf language} on $\Sigma$ is a subset of 
$\Sigma^\star$.
Given two languages $\cL_1$ and $\cL_2$ on $\Sigma$, their {\bf concatenation} is the set of sequences formed by concatenating a word
in $\cL_1$ and a word in $\cL_2$.  The set of {\bf ordered languages} on $\Sigma$ is the smallest collection of languages on $\Sigma$ that contains 
singleton languages and languages of the form $\Pi^\star$ for $\Sigma\subset \Sigma$ and is closed under finite unions and concatenations.

A {\bf norm} on $\FC$ is a function $\nu$ from the set of isomorphism classes of objects of $\FC$ to the natural numbers.
The normed category $\FC$ is said to be {\bf \boldmath{$\O$}-lingual} if, for every object $x$ of $\FC$, there exists
a finite set $\Sigma_x$ and an inclusion $\iota_x:|\FC_x|\to\Sigma_x^\star$ satisfying the following two properties:
\begin{itemize}
\item For any $\varphi:x\to y$, $\iota_x(\varphi)$ is a word of length $\nu(y)$.
\item For any ideal $I\subset |\FC_x|$, $\iota_x(I)\subset \Sigma_x^\star$ is an ordered language.
\end{itemize}
The final result that we will need is the following, which is proved in \cite[Corollary 5.3.8 and Theorem 6.3.2]{sam} (see also Corollary 8.1.4). 

\begin{theorem}\label{lingual}
Suppose that $\FC$ is endowed with a norm and an $\O$-lingual structure, $k$ is a field, and $N$ is an $\FC$-module over $k$
that is generated by the objects $x_1,\ldots,x_r$.  Let $m := \max\{|\Sigma_{x_i}|\}$ and
$$H_{\FC}(N; t) := \sum_x t^{\nu(x)}\dim N(x),$$ where the sum is over isomorphism classes of objects.
Then $H_{\FC}(N; t)$ is a rational function whose poles are contained in the set $\{1/j\mid 1\leq j\leq m\}$.
\end{theorem}

\section{Ordered surjections}
The purpose of this section is to prove theorems \ref{qGB} and \ref{B-small}.
We proceed by constructing a category $\OSB$ such that $\OSBop$ is an $\O$-lingual Gr\"obner cover of $\FSBop$.
The objects of $\OSB$ will be pairs $(E,\sigma)$, where $E$ is a totally ordered finite set and $\sigma$ is an order-reversing involution
with a unique fixed point.  We will denote the fixed point by $0$, and we will write $-e:=\sigma(e)$ for any $e\in E$.
Let
$$E^+ := \{e\in E\mid e>0\}\and E^- := \{e\in E\mid e<0\},$$
so that $$E = E^- \sqcup \{0\} \sqcup E^+.$$
For any element $e\in E$, we will write $|e| := \max\{\pm e\}$.  For any subset $D\subset E$, we will write $\init D := \min\{|e|\mid e\in S\}$.
A morphism from $(E_1,\sigma_1)$ to $(E_2,\sigma_2)$ in $\OSB$ will be a surjective map $\varphi:E_1\to E_2$ with 
$\varphi\circ\sigma_1=\sigma_2\circ \varphi$ along with the following two additional properties:
\begin{itemize}
\item[(i)] For all $e\in E_2^+$, $\init \varphi^{-1}(e) \in \varphi^{-1}(e)$.
\item[(ii)] For all $e<f\in E_2^+$, $\init \varphi^{-1}(e) < \init\varphi^{-1}(f)$.
\end{itemize}
The following lemma says that composition in $\OSB$ is well defined.

\begin{lemma}
If the maps $\varphi:(E_1,\sigma_1)\to(E_2,\sigma_2)$ and $\psi:(E_2,\sigma_2)\to(E_3,\sigma_3)$ each have
properties (i) and (ii), then so does the composition $\psi\circ\varphi:(E_1,\sigma_1)\to(E_3,\sigma_3)$.
\end{lemma}

\begin{proof}
It will suffice to check that, for all $e_3\in E_3^+$, the elements
$$e_1 := \init\varphi^{-1}\!\left(\init\psi^{-1}(e_3)\right)\and f_1 := \init (\psi\circ\varphi)^{-1}(e_3)$$ coincide.
Let $e_2:=\varphi(e_1)$ and $f_2:=\varphi(f_1)$.  Property (i) for $\varphi$ tells us that $e_2 = \init\psi^{-1}(e_3)$
and property (i) for $\psi$ tells us that $\psi(e_2) = e_3$.  Thus $(\psi\circ\varphi)(e_1) = e_3$, and therefore
$$f_1 = \init(\psi\circ\varphi)^{-1}(e_3) \leq e_1.$$
We have $\psi(f_2) = (\psi\circ\varphi)(f_1) \in \{\pm e_3\}$, 
therefore
$$e_2 = \init\psi^{-1}(e_3) = \init\psi^{-1}(\pm e_3)\leq |f_2|.$$
Applying property (ii) for $\varphi$, we find that
$$e_1 = \init\varphi^{-1}(e_2)\leq \init\varphi^{-1}(|f_2|)
= \init\varphi^{-1}(f_2) \leq f_1.$$
This completes the proof that $e_1=f_1$.
\end{proof}

Every object of $\OSB$ is isomorphic to $[-n,n]$ for some natural number $n$, and that there are no nontrivial endomorphisms.
In particular, $\OSB$ is essentially small and directed.  
Let $\Phi:\OSBop\to\FSBop$ be the forgetful functor.

\begin{lemma}\label{F}
The functor $\Phi:\OSBop\to\FSBop$ has property (F).
\end{lemma}

\begin{proof}
Unpacking the definition of property (F), we see that is is sufficient to show that, for any morphism $\varphi:(E_1,\sigma_1)\to (E_2,\sigma_2)$
in $\FSB$ and any total order of $E_2$ compatible with $\sigma_2$, there is a total order of $E_1$ compatible with $\sigma_1$
such that $\varphi$ is a morphism in $\OSB$.  Indeed, it is clear that
we can choose a total order on $E_1$, compatible with $\sigma_1$, 
with the even stronger condition that $\varphi$ is weakly order preserving.
\end{proof}

For each object $(E,\sigma)$ of $\OSB$, we define a poset structure on $E^\star$ by putting $e_1\cdots e_m\leq f_1\cdots f_n$
if there is a strictly increasing map $\theta:[m]\to[n]$ satisfying the following two conditions:
\begin{itemize}
\item For all $i\in [m]$, $e_i = f_{\theta(i)}$.
\item For all $j\in[n]$, there exists $i\in [r]$ such that $\theta(i)\leq j$ and $f_{\theta(i)}\in \{\pm f_j\}$.
\end{itemize}
In plain English, we require that $e_1\cdots e_m$ is a subword of $f_1\cdots f_n$, and that this subword contains the first occurrence of every $\sigma$ 
orbit appearing in $f_1\cdots f_n$.

\begin{proposition}\label{noetherian}
For any object, $(E,\sigma)$ of $\OSB$, the poset $E^\star$ is Noetherian.
\end{proposition}

\begin{proof}
Suppose not, and choose a sequence $w_1,w_2,w_3\ldots$ of words such that $i<j\Rightarrow w_i\not\leq w_j$.
We may assume that our sequence is {\bf minimal} in the sense that, for each $i$, 
the length of $w_i$ is minimal among all such sequences that begin $w_1,\ldots,w_{i-1}$.
Given a word $w$ and an element $e\in E$, we say that $e$ is {\bf exceptional} in $w$ if 
either $e$ or $-e$ appears exactly once in $w$ (and the other, if different, does not appear at all).
If $w$ has a non-exceptional element, we define $m(w)$ to be the number of letters appearing to the right of the last non-exceptional element.

There are only finitely many words of each length, thus we may choose a natural number $i_0$ such that, for all $i\geq i_0$, 
the length of $w_i$ is strictly greater than $\nu(E,\sigma)+1$.  It follows that, for all $i\geq i_0$, $w_i$ has a non-exceptional element.
There are only finitely many possible values for $m(w_i)$ and only finitely many elements in $E$, so we may 
find a natural number $m$ and an element $e\in E$ and 
pass to a subsequence $w_{i_1},w_{i_2},w_{i_3},\ldots$ such that $m(w_{i_j}) = m$ for all $j$ 
and the last non-exceptional element appearing in $w_{i_j}$ is $e$ for all $j$.

Let $v_{j}$ be the word obtained from $w_{i_j}$ by deleting the unique appearance of $e$, and note that $v_{j}< w_{i_j}$ for all $j$.
Consider the sequence $w_1,w_2,\ldots,w_{i_1-1},v_{1},v_{2},\ldots$.  By minimality of our original sequence, this sequence must contain
a pair of elements with the first less than or equal to the second.  We know that this cannot happen in the first $i_1-1$ terms, and we also
cannot have $w_k\leq v_{j}$ for some $k<i_1$ and $j\geq 1$, because this would imply that $w_k<w_{i_j}$.  Finally, there cannot exist
$j<k$ such that $v_{j}\leq v_{k}$, because this would imply that $w_{i_j}<w_{i_k}$.  Thus we have arrived at a contradiction.
\end{proof}

\begin{corollary}\label{ordered}
For any object, $(E,\sigma)$ of $\OSB$, every ideal in the poset $E^\star$ is an ordered language.
\end{corollary}

\begin{proof}
If $w = e_1\cdots e_n\in E^\star$, we define $I_w$ to be the principal ideal consisting of all words greater than or equal to $w$.
By Proposition \ref{noetherian}, every ideal in $E^\star$ is a finite union of principal ideals, so it is sufficient to show that $I_w$
is an ordered language.  For all $i\in [r]$, let $\Pi_i = \{\pm e_1,\ldots, \pm e_i\}$.  Then
$$I_w = e_1\Pi_1^\star e_2\Pi_2^\star\cdots e_n\Pi_n^\star$$
is a concatenation of singleton languages and languages of the form $\Pi_i^\star$, so it is ordered.
\end{proof}

Consider the norm on $\OSBop$ that takes $(E,\sigma)$ to the number of free orbits in $E$; in other words, the object $[-n,n]$ has norm $n$.
Given a morphism $\varphi:[-n,n]\to (E,\sigma)$ in $\OSB$, let
$$\iota_{(E,\sigma)}(\varphi) := \varphi(1)\cdots \varphi(n)\in E^\star.$$
Since every object of $\OSB$ is uniquely isomorphic to $[-n,n]$ for some $n$, this defines a map  
$$\iota_{(E,\sigma)}:|(\OSBop)_{(E,\sigma)}|\to E^\star.$$

\begin{lemma}\label{scratch}
Let $(E,\sigma)$ be an object of $\OSB$.
\begin{enumerate}
\item The map $\iota_{(E,\sigma)}$ is strictly order preserving.  That is, $\varphi<\varphi'\in|(\OSBop)_{(E,\sigma)}|$ if and only if $\iota_{(E,\sigma)}(\varphi)<\iota_{(E,\sigma)}(\varphi')\in E^\star$.
\item The image of an ideal in $|(\OSBop)_{(E,\sigma)}|$ is an ideal in $E^\star$.
\end{enumerate}
\end{lemma}

\begin{proof}
We begin with statement (1).  Suppose that $\varphi:[-m,m]\to(E,\sigma)$, $\varphi':[-n,n]\to(E,\sigma)$, and $\varphi<\varphi'$.
Then there exists $\psi:[-n,n]\to[-m,m]$ such that $\varphi' = \varphi\circ\psi$.  Define a map $\theta:[m]\to[n]$ by
$\theta(i) := \init \psi^{-1}(i)$.  Then $\theta$ exhibits the inequality $\iota_{(E,\sigma)}(\varphi)<\iota_{(E,\sigma)}(\varphi')\in E^\star$.

Conversely, suppose that $\iota_{(E,\sigma)}(\varphi)<\iota_{(E,\sigma)}(\varphi')\in E^\star$, and let $\theta:[m]\to[n]$ be the map
that exhibits this inequality.  By definition, for each $j\in[n]$, there exists an element $i\in [m]$ such that $\theta(i)\leq j$ and
$\varphi'(\theta(i))\in\{\pm\varphi'(j)\}$.  Let $i$ be the minimal such element.
Define $\psi(j) = i$ if $\varphi'(\theta(i)) = \varphi'(j)$ and $-i$ if $\varphi'(\theta(i)) = -\varphi'(j)$.
This extends uniquely to an $\OSB$ morphism $\psi:[-n,n]\to[m,m]$ with $\varphi'=\varphi\circ\psi$, so $\varphi<\varphi'$.

For statement (2), we first observe that the image of $\iota_{[-n,n]}$ is equal to the ideal $I_{12\cdots n}\subset [-n,n]^\star$.
Suppose that $I\subset |(\OSBop)_{[-n,n]}|$ is an ideal, $\varphi\in I$, and $w\geq \iota_{[-n,n]}(\varphi)$.
Since the image of $\iota_{[-n,n]}$ is an ideal, we have $w= \iota_{[-n,n]}(\varphi')$ for some $\varphi'$.  Statement (1) terlls us that
$\varphi<\varphi'$, so $\varphi'\in I$ and $w\in  \iota_{[-n,n]}(I)$.
\end{proof}

\begin{proposition}\label{grobling}
The category $\OSBop$ is Gr\"obner, and $\O$-lingual with respect to the maps $\iota_{(E,\sigma)}$.
\end{proposition}

\begin{proof}
Property (G2) follows from Proposition \ref{noetherian} and Lemma \ref{scratch}(1).  Property (G1) is proved by
pulling back the lexicographic order from $E^\star$ to $|(\OSBop)_{(E,\sigma)}|$.  This shows that 
$\OSBop$ is Gr\"obner.  The statement that $\OSBop$ is $\O$-lingual follows from Corollary \ref{ordered} and Lemma \ref{scratch}(2).
\end{proof}

\begin{proof}[Proof of Theorem \ref{qGB}.]
This follows from Theorem \ref{SSqG}, Lemma \ref{F}, and Proposition \ref{grobling}.
\end{proof}

\begin{proof}[Proof of Theorem \ref{B-small}.]
We begin by proving statement (1) for an $\FSBop$-module $N$ that is generated in degrees $\leq d$.
By Proposition \ref{generators} and Lemma \ref{F}, the $\OSBop$-module $\Phi^*N$ is also generated by objects of norm $\leq d$. 
Then Theorem \ref{lingual} and Proposition \ref{grobling} tell us that
$$H_{\!B}(N; t) = H_{\FSBop}(N; t) = H_{\OSBop}(\Phi^*N; t)$$ is a rational function with poles contained in the set $\{1/j\mid 1\leq j\leq 2d+1\}$.
Now suppose that $N$ is $d$-small.  By Theorem \ref{qGB}, there is some $d$ such that $N$ is finitely generated in degrees $\leq d'$,
so $H_{\!B}(N; t)$ is a rational function with poles contained in the set $\{1\leq j\leq 2d+1\}$.  However, the fact that $N$
is $d$-small means that the dimension $\dim N[-n,n]$ can only grow as fast as the dimension of a module that is finitely generated in degree $\leq d$,
therefore $H_{\!B}(N; t)$ cannot have a pole at $1/j$ when $j>2d+1$.  Finally, since passing to the associated graded of a filtration does
not change the Hilbert series of a module, this proves statement (1) when $N$ is $d$-smallish.

To prove statements (2) and (3), 
it is sufficient to check them for the principal projective $P_{[-d,d]}$.  The dimension of $P_{[-d,d]}[-n,n]$
is equal to the number of equivariant surjections from $[-n,n]$ to $[-d,d]$.  The total number of equivariant maps is $n^{2d+1}$,
and when $n$ is large, almost all equivariant maps are surjective, hence we have
$r_{\!B}^d(P_{[-d,d]}) = 1$.  
Let $\varphi$ be a morphism from $[-n,n]$ to $[-d,d]$, and consider the subgroup
$$W_\varphi\cong W_{|\varphi^{-1}(0)}| \times S_{|\varphi{-1}(1)|}\times\cdots\times S_{|\varphi^{-1}(d)|}\subset W_n$$
that stabilizes $\varphi$.  Then the $W_n$ representation $P_{[-d,d]}[-n,n]$ is isomorphic to
$$\bigoplus_\varphi \Ind_{W_\varphi}^{W_n}(\triv),$$
where the sum is over one representative of each $W_n$ orbit in $\Hom_{\FS}\!\big([-n,n],[-d,d]\big)$.
The fact that each one of these summands is a sum of representations of the form $V_{\la,\mu}$ with $\ell(\la)\leq d+1$
and $\ell(\mu)\leq d$ follows from induction on $d$ using the type $B$ Pieri rule \cite[Section 6.1.9]{GeckPfeiffer}.
\end{proof}

\section{Hyperplane arrangements}
Let $V$ be a finite dimensional vector space.
A {\bf hyperplane arrangement} in $V$ is a finite set of codimension 1 linear subspaces of $V$.
The following pair of examples will appear many times throughout this section.

\begin{example}\label{type A arrangement}
Given a nonempty finite set $E$ and any element $e\in E$, let $x_e$ be the $e^\text{th}$ coordinate function on $\C^E$,
and let $V_E \subset \C^E$
be the codimension 1 subspace consisting of vectors whose coordinates add to zero.
For any unordered pair of distinct elements $e\neq f\in E$, consider the hyperplane
$$H_{ef} := \left\{v\in V_E\mid x_e(v) = x_f(v)\right\}.$$
Let $$\cA_E := \left\{H_{ef}\mid e\neq f\in E\right\}$$
be the corresponding hyperplane arrangement in $V_E$.  When $E = [n]$, $\cA_E$ can be identified
with the Coxeter arrangement of type $A_{n-1}$, or equivalently the set of reflection hyperplanes for the Coxeter group $S_n$.
\end{example}

\begin{example}\label{type B arrangement}
For any object $(E,\sigma)$ of $\FSB$, consider the vector space
$$V_{(E,\sigma)} := \left\{v\in\C^E\mid x_e(v) + x_{\sigma(e)}(v) = 0\; \text{for all $e\in E$}\right\}\subset V_E \subset \C^E.$$
For each unordered pair $e\neq f\in E$, let 
$$J_{ef}:= V_{(E,\sigma)}\cap H_{ef}\subset V_{(E,\sigma)}.$$
Note that we have $J_{\sigma(e)\sigma(f)} = J_{ef}$ for all $e\neq f\in E$, and if $0\in E$ is the unique fixed point, then 
$J_{e\sigma(e)} = J_{e0}$ for all $e\neq 0$.
Let $$\cA_{(E,\sigma)} := \{J_{ef}\mid e\neq f\in E\}$$
be the corresponding hyperplane arrangement in $V_{(E,\sigma)}$.  When $(E,\sigma) = [-n,n]$, $\cA_{(E,\sigma)}$ can be identified
with the Coxeter arrangement of type $B_n$, or equivalently the set of reflection hyperplanes for the Coxeter group $W_n$.
\end{example}

Given a hyperplane arrangement $\cA$ in $V$, a {\bf flat} of $\cA$ is a linear subspace $F\subset V$ obtained
by intersecting some subset of the hyperplanes.
The {\bf contraction} of $\cA$ at $F$ is the hyperplane arrangement
$$\cA^F := \{F\cap H \mid F\not\subset H\in \cA\}$$
in the vector space $F$.  
The {\bf localization} of $\cA$ at $F$ is the hyperplane arrangement
$$\cA_F := \{H/F \mid F\subset H\in \cA\}$$
in the vector space $V/F$.  If $\cA_1$ is a hyperplane arrangement in $V_1$ and
$\cA_2$ is a hyperplane arrangement in $V_2$, the {\bf product} $\cA_1\times\cA_2$ is defined to be the hyperplane
arrangement in $V_1\oplus V_2$ with hyperplanes $$\{H_1\oplus V_2\mid H_1\in\cA_1\}\cup\{V_1\oplus H_2\mid H_2\in\cA_2\}.$$

\begin{example}\label{FS-flats}
For any surjective map $\varphi:E_1\to E_2$ of finite sets, we may define a flat
$$F_\varphi := \bigcap_{\substack{e\neq f\in E_1\\ \varphi(e)=\varphi(f)}} H_{ef} \subset V_{E}$$
of the arrangement $\cA_E$.
Every flat of $\cA_{E_1}$ is of this form, and if we have two surjections $\varphi:E_1\to E_2$ and $\varphi':E_1\to E_2'$,
then $F_{\varphi} = F_{\varphi'}$ if and only if there is a bijection $\psi:E_2\to E_2'$ such that $\varphi'=\psi\circ\varphi$.
The contraction of $\cA_{E_1}$ at $F_{\varphi}$ can be canonically identified with
$\cA_{E_2}$, and the localization of $\cA_{E_1}$ at the flat $F_\varphi$ can be canonically identified with the product
$$\prod_{e\in E_2} \cA_{\varphi^{-1}(e)}.$$
\end{example}

\begin{example}\label{FSB-flats}
Given a morphism $\varphi:(E_1,\sigma_1)\to(E_2,\sigma_2)$ in $\FSB$, we may define a flat
$$G_\varphi := \bigcap_{\substack{e\neq f\in E_1\\ \varphi(e)=\varphi(f)}} J_{ef} \subset V_{(E_1,\sigma_1)}$$
of the arrangement $\cA_{(E_1,\sigma_1)}$.
Every flat of $\cA_{(E_1,\sigma_1)}$ is of this form, 
and if we have two morphisms $\varphi:(E_1,\sigma_1)\to(E_2,\sigma_2)$ and $\varphi':(E_1,\sigma_1')\to(E_2,\sigma_2')$,
then $G_{\varphi} = G_{\varphi'}$ if and only if there is an isomorphism $\psi:(E_2,\sigma_2)\to (E_2',\sigma_2')$ 
such that $\varphi'=\psi\circ\varphi$.
The contraction of $\cA_{(E_1,\sigma_1)}$ at 
$G_{\varphi}$ can be canonically identified with
$\cA_{(E_2,\sigma_2)}$.
To understand the localization, we first choose a decomposition
$$E_2 = P_2 \sqcup \{0\} \sqcup \sigma_2(P_2),$$
where $0\in E_2$ is the unique fixed point.  Then the localization of $\cA_{(E_1,\sigma_1)}$ at the flat $G_\varphi$ 
can be canonically identified with the product
$$\cA_{(\varphi^{-1}(0),\sigma_1)}\times \prod_{e\in P_2} \cA_{\varphi^{-1}(e)}.$$
\end{example}

\begin{remark}
If we want to avoid choosing a decomposition of $E_2$, we can replace the product over $P_2$ with a product over non-fixed
$\sigma_2$-orbits, and replace the preimage of $e\in P_2$ with the set of $\sigma_1$-orbits in the preimage of the $\sigma_2$-orbit.
This would be more canonical, but also more unwieldy to notate.
\end{remark}


\section{Orlik--Solomon algebras}\label{sec:OS}
Let $\cA$ be a hyperplane arrangement.
A set $\mathcal{D}\subset \cA$ is called {\bf dependent} if the codimension of its intersection
is smaller than its cardinality (equivalently, if the corresponding set of normal vectors is linearly dependent).  
For any dependent set $\mathcal{D} = \{H_1,\ldots,H_k\}\subset \cA$ of cardinality $k$,
we define a class $$\partial u_\mathcal{D} := \sum_{i=1}^k (-1)^i \prod_{j\neq i} u_{H_j}$$
in the exterior algebra
$\Lambda_\C[u_H\mid H\in \cA]$.  Note that the element $u_S$ as we have defined it depends on the ordering of the elements of $S$,
but only up to sign.  The {\bf Orlik--Solomon algebra} $\OS(\cA)$\footnote{It is typical to denote the Orlik--Solomon algebra
either $OS(\cA)$ or $A(\cA)$, but we wish to avoid conflict with the notation for the category $\OSB$ and with the use of the letter $A$
for type $A$ structures.  So, with apologies to Peter Orlik, we are just using the letter $S$.}
is defined as the quotient of $\Lambda_\C[u_H\mid H\in \cA]$
by the ideal generated by $\partial u_S$ for every dependent set $\mathcal{D}$.
If $\cA_1$ and $\cA_2$ are two hyperplane arrangements, then
\begin{equation}\label{kunneth}\OS(\cA_1\times\cA_2) \cong \OS(\cA_1)\otimes\OS(\cA_2).\end{equation}
If $F$ is a flat of $\cA$, then there is a canonical map $$\OS(\cA)\to \OS(\cA^F)$$
defined by sending $u_H$ to $u_{F\cap H}$ if $F\not\subset H$ and to zero otherwise.

\begin{remark}
If $V$ is a vector space over $\C$, then $\OS(\cA)$
is canonically isomorphic to the cohomology
of the complement of $\cA$ \cite{OS}.  In this case, Equation \eqref{kunneth}
can be regarded as an application of the K\"unneth theorem.  For a topological interpretation of the map
from $\OS(\cA)$ to $\OS(\cA^F)$, see \cite[Section 3]{fs-braid}.
\end{remark}

Fix a natural number $i$.
By Example \ref{FS-flats}, we have an $\FSA$-module that assigns to a finite set $E$ the vector space $\OS^i\!\left(\cA_E\right)$,
and to a surjection $\varphi:E_1\to E_2$ the map
$$\OS^i\!\left(\cA_{E_1}\right) \to \OS^i\!\left((\cA_{E_1})^{F_\varphi}\right) \cong \OS^i\!\left(\cA_{E_2}\right).$$
We denote this module by $\OS^i_{\!A}$, and we denote the dual $\FSAop$-module by $(\OS^i_{\!A})^*$.
Similarly, by 
Example \ref{FSB-flats}, we have an $\FSB$-module that assigns to an object $(E,\sigma)$ the vector space $\OS^i\!\left(\cA_{(E,\sigma)}\right)$,
and to a morphism $\varphi:(E_1,\sigma_1)\to (E_2,\sigma_2)$ the map
$$\OS^i\!\left(\cA_{(E_1,\sigma_1)}\right) \to \OS^i\!\left((\cA_{(E_1,\sigma_1)})^{G_\varphi}\right) \cong \OS^i\!\left(\cA_{(E_2,\sigma_2)}\right).$$
We denote this module by $\OS^i_{\!B}$, and we denote the dual $\FSBop$-module by $(\OS^i_{\!B})^*$.

\begin{proof}[Proof of Proposition \ref{OSB-prop}.]
We have $(\OS^0_B)^* \cong Q_0$, so the first statement is trivial, and we may assume that $i>0$.
Since the Orlik--Solomon algebra is generated in degree 1, $\OS^i_{\!B}$ is a quotient
of $(\OS^1_B)^{\otimes i}$, and therefore $(\OS^i_{\!B})^*$ is a submodule of $\big((\OS^1_B)^*\big)^{\otimes i}$.
Thus it will suffice to show that, for any object of $\FSB$ with at least $2i$ free orbits, every element of 
$(\OS^1_B)^*(E,\sigma)^{\otimes i}$ is a linear combination of pullbacks of classes 
along various maps to smaller objects.  

Let $0\in E$ denote the unique fixed point.
The vector space $\OS^1_B(E,\sigma)$ is spanned by the elements $u_{ef}$ for unordered pairs $e\neq f$ that
are distinct from $0$ (recall that we have $u_{e0} = u_{e\sigma(e)}$ for any $e\neq 0$).  For such an unordered pair,
let $v_{ef}\in \OS^1_B(E,\sigma)^*$ be the element that evaluates to 1 on $u_{ef} = u_{\sigma(e)\sigma(f)}$ and to 0 on all other generators.
Then $(\OS^1_B)^*(E,\sigma)^{\otimes i}$ is spanned by classes of the form $v_{e_1f_1}\otimes\cdots \otimes v_{e_if_i}$.

Let $$F := \{e_1,\sigma(e_1),f_1,\sigma(f_1)\ldots,e_i,\sigma(e_i),f_i,\sigma(f_i),0\}\subset E,$$
so that $(F,\sigma)$ is an object of $\FSB$ with at most $2i$ free orbits.
Define a morphism $\varphi:(E,\sigma)\to(F,\sigma)$ by fixing $F\subset E$ and sending $E\smallsetminus F$ to $0$.
Our hypothesis implies that the class $v_{e_1f_1}\otimes\cdots \otimes v_{e_if_i}$ is sent to itself by the map
$$\varphi^*:(\OS^1_B)^*(F,\sigma)^{\otimes i} \to (\OS^1_B)^*(E,\sigma)^{\otimes i}.$$
If the cardinality of $F$ is strictly smaller than $2i+1$, then we are done.  If not, then the classes appearing in the definition of $F$ are all distinct,
so we may assume for ease of notation that $(F,\sigma)=[-2i,2i]$,
with $e_j = (2j-1)$ and $f_i = 2j$ for all $j$.

We will consider three morphisms $\psi_1,\psi_2,\psi_3$ from $[-2i,2i]$ to $[1-2i,2i-1]$,
and we will prove that the class $$v_{12}\otimes\cdots \otimes v_{2i-3,2i-2}\otimes v_{2i-1,2i} \in (\OS^1_B)^*[-2i,2i]^{\otimes i}$$
is in the span of the images of the pullbacks along these three morphisms.
Each of these morphisms will fix $[2-2i,2i-2]$, 
and they will be defined on the elements $\{2i-1,2i\}$ as follows:\footnote{Note that this determines what the morphisms do to the elements $\{-2i,1-2i\}$.}
\begin{itemize}
\item $\psi_1(2i) = 2i-1$ and $\psi_1(2i-1) = 1-2i$
\item $\psi_2(2i) = 2i-1$ and $\psi_2(2i-1) = 0$
\item $\psi_3(2i) = 0$ and $\psi_3(2i-1) = 2i-1$.
\end{itemize}
For each positive integer $j<i$, all three of these maps send the class $v_{2j-1,2j}$ to itself.  Furthermore, we have
\begin{eqnarray*}
\psi_1^*\left(v_{2i-1,1-2i}\right) &=& v_{2i-1,1-2i} + v_{2i,-2i} + v_{2i-1,2i}\\
\psi_2^*\left(v_{2i-1,1-2i}\right) &=& v_{2i-1,1-2i} + v_{2i-1,-2i} + v_{2i-1,2i}\\
\psi_3^*\left(v_{2i-1,1-2i}\right) &=& v_{2i,-2i} + v_{2i-1,-2i} + v_{2i-1,2i}
\end{eqnarray*}
and therefore $$\psi_1^*\left(v_{2i-1,1-2i}\right) - \psi_2^*\left(v_{2i-1,1-2i}\right) + \psi_3^*\left(v_{2i-1,1-2i}\right) = v_{2i-1,2i}.$$
It follows that we have
\begin{eqnarray*}
v_{12}\otimes\cdots \otimes v_{2i-3,2i-2}\otimes v_{2i-1,2i} &=&
\psi_1^*\left(v_{12}\otimes\cdots \otimes v_{2i-3,2i-2}\otimes v_{2i-1,1-2i}\right)\\
&& - \psi_2^*\left(v_{12}\otimes\cdots \otimes v_{2i-3,2i-2}\otimes v_{2i-1,1-2i}\right)\\
&& + \psi_2^*\left(v_{12}\otimes\cdots \otimes v_{2i-3,2i-2}\otimes v_{2i-1,1-2i}\right).
\end{eqnarray*}
This completes the proof of smallness.  For the final statement, we note that
$\dim \OS_{\!B}^1[-n,n] = n^2$, therefore 
$\dim \OS_{\!B}^i[-n,n] \leq \binom{n^2}{i}$, and
$$\lim_{n\to \infty} \frac{\binom{n^2}{i}}{(2i-1)^n} = 0.$$
Thus $r_{2i-1}\!\left((\OS^i_{\!B})^*\right)=0$.
\end{proof}

\section{Combining small modules}
We begin with the following analogue of \cite[Lemma 4.2]{fs-braid}, which mixes modules over $\FSAop$ and $\FSBop$.

\begin{lemma}\label{composition}
Let $N$ be an $\FSBop$-module and let $M_1,\ldots,M_p$ be $\FSAop$-modules, with $N$ $d$-small and $M_i$ $c_i$-small for all $i$.
Consider the $\FSBop$-module $R$ defined on objects by the formula
$$R(E,\sigma) = \bigoplus_{\varphi:(E,\sigma)\to[-p,p]} N(\varphi^{-1}(0),\sigma)\otimes M_1(\varphi^{-1}(1))\otimes\cdots\otimes M_p(\varphi^{-1}(p)),$$
with maps defined in the natural way.  The module $R$ is $(d+c_1+\cdots+c_p)$-small.
\end{lemma}

\begin{proof}
Since smallness is preserved by taking direct sums and passing to subquotients, we may immediately reduce to the case where
$N$ is the principal projective $P_{[-n,n]}$ for some $n\leq d$ and for each $i$, $M_i$ is the principal projective $P_{[m_i]}$ for some $m_i\leq c_i$.
Then \begin{eqnarray*}
R(E,\sigma) &\cong& \bigoplus_{\varphi:(E,\sigma)\to[-p,p]} N(\varphi^{-1}(0),\sigma)\otimes M_1(\varphi^{-1}(1))\otimes\cdots\otimes M_p(\varphi^{-1}(p))\\
&\cong& \bigoplus_{\varphi:(E,\sigma)\to[-p,p]} \C\left\{\Hom_{\FSB}\!\big((\varphi^{-1}(0),\sigma),[-n,n]\big)\times
\prod_{i=1}^p \Hom_{\FSA}\!\big(\varphi^{-1}(i),[m_i]\big)\right\}\\
&\cong& \C\Big\{\Hom_{\FSB}\!\big((E,\sigma), [-(n+m_1+\cdots+m_p),(n+m_1+\cdots+m_p)]\big)\Big\}\\
&\cong& P_{[-(n+m_1+\cdots+m_p),(n+m_1+\cdots+m_p)]}(E,\sigma).
\end{eqnarray*}
Thus $R$ is $(n+m_1+\cdots+m_p)$-small, and therefore $(d+c_1+\cdots+c_p)$-small.
\end{proof}

For any natural numbers $p$ and $i$ and any object $(E,\sigma)$ of $\FSB$, let
\begin{eqnarray*}
C_{p,i}(E,\sigma) &:=& \bigoplus_{\varphi:(E,\sigma)\to[-p,p]} \OS^i\!\left((\cA_{(E,\sigma)})_{G_\varphi}\right)\\
&\cong& \bigoplus_{\varphi:(E,\sigma)\to[-p,p]} \OS^i\!\left(\cA_{(\varphi^{-1}(0),\sigma)}\times\cA_{\varphi^{-1}(1)}\times\cdots\times\cA_{\varphi^{-1}(p)}\right)\\
&\cong& \bigoplus_{\varphi:(E,\sigma)\to[-p,p]} \Big(\OS\!\left(\cA_{(\varphi^{-1}(0),\sigma)}\right) \otimes \OS\!\left(\cA_{\varphi^{-1}(1)}\right)
\otimes\cdots\otimes\OS\!\left(\cA_{\varphi^{-1}(p)}\right)\Big)^i\\
&=& \bigoplus_{\substack{\varphi:(E,\sigma)\to[-p,p]\\ i_0+i_1+\cdots+i_p=i}}
\OS^{i_0}_B\!\left(\varphi^{-1}(0),\sigma\right)\otimes \OS_{\!A}^{i_1}\!\left(\varphi^{-1}(1)\right)\otimes\cdots\otimes\OS_{\!A}^{i_p}\!\left(\varphi^{-1}(p)\right).
\end{eqnarray*}
Then $C_{p,i}$ is naturally an $\FSB$-module, and its dual $C_{p,i}^*$ is an $\FSBop$-module.
The following proposition is the type $B$ analogue of \cite[Proposition 5.3]{fs-braid}, and will be needed in the next section for the proof of Theorem \ref{KLB-thm}.

\begin{proposition}\label{C-small}
If $i>0$, the $\FSBop$-module $C_{p,i}^*$ is $(2i-1+p)$-small.  If $i=0$, it is $p$-small.
\end{proposition}

\begin{proof}
By Propositions \ref{OSA-prop} and \ref{OSB-prop} and Lemma \ref{composition}, 
the direct summand of $C_{p,i}^*$ corresponding to the tuple $(i_0,i_1,\ldots,i_p)$ is 
$(2i-1+d)$-small,
where $d$ is the number of $k\in\{0,1,\ldots,p\}$ such that $i_k=0$.  If $i>0$, the maximum possible value of $d$ is $p$, 
so the entire sum is $(2i-1+p)$-small.  If $i=0$, then $d=p+1$, and the sum is $p$-small.
\end{proof}

\section{Kazhdan--Lusztig coefficients}\label{sec:KL}
Let $V$ be a vector space over $\C$ and $\cA$ a hyperplane arrangement in $V$ with $\bigcap_{H\in\cA} H = \{0\}$.
We have an inclusion
$$V\to \prod_{H\in\cA} V/H \cong \prod_{H\in\cA} \mathbb{A}^1 \subset \prod_{H\in\cA} \mathbb{P}^1.$$
Let $Y_\cA$ be the closure of $V$ inside of the product of projective lines, and let $X_\cA\subset Y_\cA$ be the open subset consisting
of points where no coordinate is equal to zero.  The affine variety $X_\cA$ was introduced in \cite{PS}, and is called the {\bf reciprocal plane} of $\cA$.
We will be interested in the intersection cohomology of $X_{\cA}$ with coefficients in $\C$, 
which vanishes in odd degree, and has the property that its Poincar\'e polynomial
$$\sum_{i\geq 0} t^i \dim \IH^{2i}(X_{\cA})$$
is equal to the {\bf Kazhdan--Lusztig polynomial} of $\cA$ \cite[Proposition 3.12]{EPW}.  For this reason, we may regard the vector space 
$\IH^{2i}\!\left(X_{\cA}\right)$ as a catigorification of the $i^\text{th}$ Kazhdan--Lusztig coefficient of $\cA$.

If $F$ is a flat of $\cA$, there is a (noncanonical) inclusion of varieties $X_{\cA^F}\to X_\cA$, which induces a (canonical)
map of intersection cohomology groups $\IH^{2i}\!\left(X_\cA\right)\to \IH^{2i}\!\left(X_{\cA^F}\right)$.
These maps are functorial \cite[Theorem 3.3]{fs-braid}; in particular, we have an $\FSA$-module $D^i_{\!A}$ that takes a finite set $E$ to the vector space $\IH^{2i}\!\left(X_{\cA_E}\right)$
and a morphism $\varphi:E_1\to E_2$ to the map
$$\IH^{2i}\!\left(X_{\cA_{E_1}}\right)\to \IH^{2i}\!\left(X_{(\cA_{E_1})^{F_\varphi}}\right)\cong \IH^{2i}\!\left(X_{\cA_{E_2}}\right),$$
and we have an $\FSB$-module $D^i_{\!B}$ that takes an object $(E,\sigma)$ to the vector space $\IH^{2i}\!\left(X_{\cA_{(E,\sigma)}}\right)$
and a morphism $\varphi:(E_1,\sigma_1)\to (E_2,\sigma_2)$ to the map
$$\IH^{2i}\!\left(X_{\cA_{E_1}}\right) \to \IH^{2i}\!\left(X_{(\cA_{E_1})^{F_\varphi}}\right) \cong \IH^{2i}\!\left(X_{\cA_{E_2}}\right).$$
  
\begin{proof}[Proof of Theorem \ref{KLB-thm}.]
For any hyperplane arrangement $\cA$, there a spectral sequence $N(i,\cA)$ converging
to $\IH^{2i}\!\left(X_{\cA}\right)$ with
$$N(i,\cA)_1^{p,q} = \bigoplus_{\dim F = p}\OS^{2i-p-q}\!\left(\cA_F\right)\otimes \IH^{2(i-q)}\!\left(X_{\cA^F}\right),$$
where the direct sum is over flats $F$ of $\cA$ \cite[Theorem 3.1]{fs-braid}.
For any object $(E,\sigma)$ of $\FSB$, let $N(i,E,\sigma) = N\!\left(i,\cA_{(E,\sigma)}\right)$.
Then $N(i,E,\sigma)$ converges to $D^i_{\!B}(E,\sigma)$, and
Example \ref{FSB-flats} tells us that
$$N(i,E,\sigma)_1^{p,q} \cong \left(C_{p,2i-p-q}(E,\sigma)\otimes D^{i-q}_{\!B}(p)\right)^{W_p}.$$
This construction is functorial \cite[Theorem 3.3]{fs-braid}, meaning that we have a spectral sequence $N(i)$ in the category of $\FSB$-modules
converging to $D^i_{\!B}$
with $$N(i)_1^{p,q} = \left(C_{p,2i-p-q} \otimes D^{i-q}_{\!B}(p)\right)^{W_p}.$$
Dualizing, we obtain a spectral sequence $N^*(i)$ in the category of $\FSBop$-modules converging to $(D^i_{\!B})^*$.
Since $N(i)_1^{p,q}$ is a submodule of $C_{p,2i-p-q} \otimes D^{i-q}_{\!B}(p)$,
$N^*(i)_1^{p,q}$ is a quotient of $C^*_{p,2i-p-q} \otimes D^{i-q}_{\!B}(p)^*$, and Proposition \ref{C-small} implies that it is 
$(2(2i-p-q)-1+p)$-small unless $p+q=2i$, in which case it is $p$-small.
Furthermore, we have $D^{i-q}_{\!B}(p) = 0$ unless either $(p,q) = (0,i)$ or $p>2(i-q)$ \cite[Proposition 3.4]{EPW}.

Let us consider first the case where $p+q=2i$.  Since $i>0$, we cannot have $(p,q) = (0,i)$, so we must have $p>2(i-q)$ for 
$N^*(i)_1^{p,q}$ to be nonzero.
This means that $p$ cannot be equal to $2i$, so we have $p\leq 2i$, which implies that $N^*(i)_1^{p,q}$ is $(2i-1)$-small.
Even better, it tells us that $N^*(i)_1^{p,q}$ is $(2i-2)$-small unless $p=2i-1$ and $q=1$.

Now let us consider the case where $p+q<2i$.
If $(p,q) = (0,i)$, then $(2(2i-p-q)-1+p) = 2i-1$, so $N^*(i)_1^{0,i}$ is $(2i-1)$-small.
If $p>2(i-q)$, then $2(2i-p-q)-1+p = 2(i-q)-p+2i-1<2i-1$,
so $N^*(i)_1^{p,q}$ is $(2i-2)$-small.

Since $N^*(i)$ converges to $(D^i_{\!B})^*$ and the entries of the $E_1$-page of $N^*(i)$ are all $(2i-1)$-small, 
we can conclude that $(D^i_{\!B})^*$ is $(2i-1)$-smallish.  Furthermore, the $E_\infty$ page of $N^*(i)$
is concentrated on the diagonal $p+q=2i$, hence
\begin{eqnarray*}r_{2i-1}\!\left(\left(D^i_{\!B}\right)^*\right) 
&=& \sum_{p,q} r_{2i-1}\!\Big(N^*(i)_\infty^{p,q}\Big)\\
&=& \sum_{p,q} (-1)^{p+q} r_{2i-1}\!\Big(N^*(i)_\infty^{p,q}\Big)\\
&=& \sum_{p,q} (-1)^{p+q} r_{2i-1}\!\Big(N^*(i)_1^{p,q}\Big).
\end{eqnarray*}
Since $r_{2i-1}$ vanishes on any $\FSBop$-module that is $(2i-2)$-small, this equation simplifies to
$$r_{2i-1}\!\left(\left(D^i_{\!B}\right)^*\right) = r_{2i-1}\!\left(N^*(i)_1^{2i-1,1}\right) + (-1)^i r_{2i-1}\!\left(N^*(i)_1^{0,i}\right).$$
We have $N^*(i)_1^{0,i} = C^*_{0,i} = (\OS_{\!B}^i)^*$, thus Proposition \ref{OSB-prop} says that $r_{2i-1}\!\left(N^*(i)_1^{0,i}\right) = 0$.
Finally, we have
$$N(i)_1^{2i-1,1} = \left(C_{2i-1,0}\otimes D^{i-1}_{\!B}[1-2i,2i-1]\right)^{W_{2i-1}}
\cong \left(P_{[1-2i,2i-1]}[-n,n]\otimes D^{i-1}_{\!B}[1-2i,2i-1]\right)^{W_{2i-1}}
,$$
so 
$$\dim N^*(i)_1^{2i-1,1}[-n,n] = \dim N(i)_1^{2i-1,1}[-n,n]
= \frac{\dim P_{[1-2i,2i-1]}[-n,n] \cdot \dim D^{i-1}_{\!B}[1-2i,2i-1]}{|W_{2i-1}|},
$$
where the last equality follows from the fact that the group $W_{2i-1}$ 
acts freely on a basis for $P_{[1-2i,2i-1]}[-n,n]$.
We therefore have
$$r_{2i-1}\!\left(\left(D^i_{\!B}\right)^*\right) = r_{2i-1}\!\big(P_{[1-2i,2i-1]}\big)\cdot \frac{\dim D^{i-1}_{\!B}[1-2i,2i-1]}{|W_{2i-1}|} = \frac{\dim D^{i-1}_{\!B}[1-2i,2i-1]}{|W_{2i-1}|}.$$
This completes the proof.
\end{proof}

\begin{example}
We illustrate Theorems \ref{B-small} and \ref{KLB-thm} when $i=1$.
The coefficient of $t$ in the Kazhdan--Lusztig polynomial of a hyperplane arrangement $\cA$ is equal to the number of flats
of dimension 1 minus the number of hyperplanes \cite[Proposition 2.12]{EPW}, thus Example \ref{FSB-flats} tells us that
$$\dim D^{1}_{\!B}[-n,n] = \Big{|}\Hom_{\FSB}\!\big([-n,n],[-1,1]\big)/{W_1}\Big{|} - n^2 = \frac{3^n-1}{2} - n^2.$$
This means that $$H_{\!B}\!\left({(D^{1}_{\!B})^*}, t\right) = \sum_{n=0}^\infty\left(\frac{3^n-1}{2} - n^2\right) t^n = 
\frac{1}{2(1-3t)} - \frac{1}{2(1-t)} - \frac{t}{(1-t)^2} - \frac{2t^2}{(1-t)^3}.$$
This is a rational function with a poles at $1$ and $1/3$.  The pole at $1/3$ is simple, with residue
$$\frac{1}{2} = \frac{\dim D^{0}_{\!B}[-1,1]}{|W_{1}|}.$$
As a representation of $W_n$, $D^{1}_{\!B}[-n,n]^* \cong D^{1}_{\!B}[-n,n]$ is isomorphic to the permutation representation
with basis given by the flats of dimension 1 modulo the permutation representation with basis given by the hyperplanes \cite[Corollary 2.10]{GPY}.
If $n<3$, then $D^{1}_{\!B}[-n,n] = 0$, while if $n\geq 3$, using the branching rule in \cite[Lemma 6.1.3]{GeckPfeiffer} allows us to compute
$$D^{1}_{\!B}[-n,n] = \bigoplus_{\substack{|\la| \leq n\\ \ell(\la)\leq 2}} V_{\la,[n-|\la|]}^{\oplus c_\la},$$
where $$c_\la = \begin{cases}
\lfloor \la_1/2\rfloor -1& \text{if $\la = [n]$ or $\la = [n-1,1]$}\\
\lfloor \la_1/2\rfloor & \text{if $\la = [n-2,2]$ or $\la = [n-2]$}\\
\lfloor \la_1/2\rfloor + 1& \text{otherwise.}
\end{cases}$$
\end{example}

\bibliography{./symplectic}
\bibliographystyle{amsalpha}

\end{document}